\documentclass[11pt]{amsart}

\usepackage[english]{babel}
\usepackage[letterpaper,top=2cm,bottom=2cm,left=3cm,right=3cm,marginparwidth=1.75cm]{geometry}

\usepackage{amsmath,amssymb,amsfonts,amsthm}
\usepackage{graphicx}
\usepackage{tikz}
\usepackage[colorlinks=true, allcolors=blue]{hyperref}
\usepackage{enumerate,xcolor}
\usepackage{indentfirst}
\usepackage{comment}

\newtheorem{theorem}{Theorem}

\newtheorem{corollary}[theorem]{Corollary}
\newtheorem{proposition}[theorem]{Proposition}
\newtheorem{lemma}[theorem]{Lemma}
\newtheorem{example}[theorem]{Example}
\newtheorem{remark}[theorem]{Remark}
\newtheorem{question}[theorem]{Question}

\DeclareMathOperator{\Der}{Der}
\DeclareMathOperator{\Aut}{Aut}
\DeclareMathOperator{\Spec}{Spec}
\DeclareMathOperator{\ML}{ML}
\DeclareMathOperator{\Pol}{Pol}
\DeclareMathOperator{\Bir}{Bir}

\newcommand{\K}{\Bbbk}

\newcommand{\id}{\operatorname{id}}
\newcommand{\D}{\mathcal{D}}

\title[Simple derivations and isotropy on Danielewski-type algebras]
{Simple derivations and isotropy on Danielewski-type algebras}

\author{Rene Baltazar}
\date{}

\begin{document}

\begin{abstract}
Let $\K$ be an algebraically closed field of characteristic zero. We study
isotropy groups of simple derivations on Danielewski-type algebras
\[
A_{c,q}=\K[x,y,z]/(c(x)z-q(x,y)).
\]
More precisely, motivated by the recent result of Mendes and Pan for simple derivations of $\K[x,y]$ (\cite[Theorem 1]{MP}), we ask whether every simple derivation $D$ of $A_{c,q}$ has trivial isotropy group.

We prove that this property holds generically: for each fixed pair of degrees $\deg(c)\geq2$
and $\deg_y(q)\geq2$, there exists a nonempty Zariski open subset of the parameter space of reduced
pairs $(c,q)$ such that every simple derivation of $A_{c,q}$ has trivial isotropy group. On the other hand, we construct a special Danielewski-type algebra
admitting a simple derivation with nontrivial isotropy. Thus, the Mendes-Pan phenomenon holds generically for Danielewski-type algebras, but fails in full generality. Over
$\K=\mathbb C$, we also discuss the associated foliations and formulate questions about their polynomial and birational symmetries.
\end{abstract}

\maketitle

\section{Introduction}

Let $\K$ be an algebraically closed field of characteristic zero. A derivation $D$ of a $\K$-algebra $A$ is called simple if the only $D$-stable ideals are the trivial.
Geometrically, when $A$ is the coordinate ring of an affine variety, this
means that the vector field defined by $D$ has no proper invariant algebraic
subset. The study of simple derivations has its origins in the papers of Seidenberg \cite{Seidenberg1967}, Hart \cite{Hart1975}, and Shamsuddin \cite{Shamsuddin1977,Shamsuddin1982}. These works initiated a line of research connecting differential ideals, polynomial vector fields, and the algebraic structure of affine varieties.

The group $\Aut(A)$ of all $\K$-algebra automorphisms of $A$ acts on $\Der(A)$, the $\K$-vector space of all $\K$-derivations of $A$, by conjugation: for $\sigma\in\Aut(A)$ and $D\in\Der(A)$, $\sigma\cdot D=\sigma D\sigma^{-1}$. The \textit{isotropy} group of $D$ with respect to this action is
\[
\Aut_D(A)
=
\{\sigma\in\Aut(A)\mid \sigma D=D\sigma\}.
\]

Equivalently,  $\Aut_D(A) $ is the automorphism group of the differential  $\K $-algebra  $(A,D) $. This point of view was studied for polynomial differential rings in two variables in \cite{PanBalta}.

In the case of the polynomial algebra $\K[x,y]$, a recent result of Mendes and Pan asserts that if $D$ is a simple derivation, then its isotropy group is trivial (\cite[Theorem 1]{MP}). This result extends the earlier case of simple Shamsuddin derivations, for which the triviality of the isotropy group was proved in \cite[Theorem 6]{Baltazar2016}.

The purpose of this paper is to investigate an analogue of the Mendes--Pan
phenomenon for Danielewski-type algebras. More precisely, we consider algebras
of the form
\[
A_{c,q}
=
\K[x,y,z]/(c(x)z-q(x,y)),
\]
where $c(x)\in\K[x]$ has degree at least two and $q(x,y)\in\K[x,y]$ is
quasi-monic with respect to $y$. These algebras provide a natural extension
of the affine-plane setting. The affine plane occurs in the degenerate case
$c(x)\in\K^\ast$, while the case $\deg(c)\geq2$ gives genuinely new
affine surfaces, with different automorphism group.

We are interested in the following question:

\begin{question}
Let $D\in\Der(A_{c,q})$ be a simple derivation. Is it true that
\[
\Aut_D(A_{c,q})=\{\id\}?
\]
\end{question}

In Section \ref{Preliminares}, we recall the necessary background on Danielewski-type algebras, their reduced form, and the structure of their automorphism groups. 

In Section \ref{3}, we prove the main reduction to the diagonal group. More precisely, we show that the isotropy group of a simple derivation has trivial intersection
with the unipotent subgroup $U_{\K}(A_{c,q})$. Consequently, the canonical
homomorphism $\psi:\Aut(A_{c,q})\longrightarrow G_{c,q}$
restricts to an injection $\Aut_D(A_{c,q})\hookrightarrow G_{c,q}$. We also formulate a fixed-point criterion in terms of the map $\psi$.

In Section \ref{4}, we treat the generic case in degree $d=\deg_y(q)\geq3$. We introduce a lattice $\Lambda(c,q)\subseteq\mathbb Z^2$ determined by the
monomials appearing in $c$ and $q$, and prove that $\Lambda(c,q)=\mathbb Z^2$ is sufficient to
ensure that $G_{c,q}=\{1\}$. We then exhibit an explicit nonempty Zariski open subset on which this lattice condition holds. This proves the generic triviality of the isotropy group in
degree $d\geq3$.

In Section \ref{5}, we study the quadratic case $d=\deg_y(q)=2$. Although $G_{c,q}$ is not trivial in this case, we prove that, on a suitable Zariski open subset of the parameter space, every lift of the nontrivial
diagonal element fixes a maximal ideal. This gives the generic triviality of the isotropy group in degree two. Together with Section 4, this proves that
the Mendes-Pan phenomenon holds generically for Danielewski-type algebras of every degree $d\geq2$.

In Section \ref{6}, we show that the generic result cannot be extended to all Danielewski-type algebras. We consider
\[
\mathbb A=\K[x,y,z]/(x^2z-y^2-1),
\]
which has a fixed-point-free diagonal involution. We then construct a simple derivation on the invariant quotient algebra and lift it to a simple derivation of $\mathbb A$ commuting with this involution. This gives a simple derivation with nontrivial isotropy, and therefore shows that the analogue of the Mendes-Pan theorem fails in full generality.

Finally, in Section \ref{7}, over $\K=\mathbb C$, this problem also admits a geometric formulation in terms of foliations. A derivation $D$ of $A_{c,q}$ defines an algebraic foliation on the smooth locus of $X=\Spec(A_{c,q})$. Thus one may study not
only automorphisms commuting with $D$, but also automorphisms preserving the associated foliation. This viewpoint is inspired by the work of Cousin, Mendes and Pan on foliations associated to simple derivations of
$\mathbb C[x,y]$; see \cite{CMP}. Motivated by this perspective, we conclude the paper by proposing several questions about polynomial and birational symmetries of these foliations in the Danielewski-type setting.

\section{Preliminaries}\label{Preliminares}

In this section, we recall some structural facts on Danielewski-type algebras and fix the notation used throughout the paper. Furthermore, $\K$ denotes an algebraically closed field of characteristic zero. Let
\[
A_{c,q}
=
\K[x,y,z]/(c(x)z-q(x,y)),
\]
where $c(x)\in\K[x]$ is nonconstant and $q(x,y)\in\K[x,y]$ is quasi-monic of degree $d\geq 2$ with respect to $y$. We denote by $\bar x,\bar y,\bar z$ the images of $x,y,z$ in $A_{c,q}$. The algebra $A_{c,q}$ carries the locally nilpotent derivation
\[
\xi_{c,q}(\bar x)=0,
\quad
\xi_{c,q}(\bar y)=c(\bar x),
\quad
\xi_{c,q}(\bar z)=\partial_yq(\bar x,\bar y).
\]
We call $\xi_{c,q}$ the standard locally nilpotent derivation of $A_{c,q}$.

Let $R$ be a ring and let $p(y)\in R[y]$ be a quasi-monic polynomial of degree $d\geq 1$. We say that $p(y)$ is centered if the coefficient of $y^{d-1}$ is zero. We recall also the reduced form used in \cite{ABEG}. The algebra $A_{c,q}$ is said to be in reduced form if $c(x)$ is centered, $q(x,y)$ is centered as a polynomial in $\K[x,y]$ and $\deg_x(q)<\deg(c)$.

By \cite[Proposition 4]{ABEG}, every Danielewski-type algebra $A_{c,q}$ with $c(x)$ nonconstant and $q(x,y)$ quasi-monic of degree at least two with respect to $y$ is isomorphic to an algebra $A_{\widetilde c,\widetilde q}$
in reduced form. Moreover, this isomorphism is compatible with the standard locally nilpotent derivations.

Since we study pairs $(A_{c,q},D)$, when passing to reduced form we also transport the derivation. If
\[
\Phi:A_{c,q}\longrightarrow A_{\widetilde c,\widetilde q}
\]
is such an isomorphism, we denote $\widetilde D=\Phi D\Phi^{-1}$. Then $D$ is simple if and only if $\widetilde D$ is simple. Moreover, conjugation by $\Phi$ gives an isomorphism of isotropy groups
\[
\Aut_D(A_{c,q})
\cong
\Aut_{\widetilde D}(A_{\widetilde c,\widetilde q}),
\quad
\sigma\longmapsto \Phi\sigma\Phi^{-1}.
\]
Thus, the study of isotropy groups of simple derivations is unchanged after passing to reduced form. So, we assume that $A_{c,q}$ is in reduced form and write
\[
n=\deg(c),
\quad
d=\deg_y(q).
\]

We use the fact that, when $n=\deg(c)\geq 2$, the Makar-Limanov invariant of
$A_{c,q}$ is $\ML(A_{c,q})=\K[\bar x]$ (see, \cite[Theorem 7]{BianchiVeloso}). In particular, every automorphism of $A_{c,q}$ preserves the subalgebra $\K[\bar x]$. This is one of the standard structural properties of Danielewski-type algebras; in the present form, we refer to \cite{ABEG}.

We denote by $\mathbb G_a$ and $\mathbb G_m$ the additive and multiplicative
algebraic groups over $\K$, respectively. Following \cite{ABEG}, we consider
the subgroup $G_{c,q}$ of $\mathbb G_m^2$, the $2$-dimensional algebraic torus over $\K$, defined by
\[
G_{c,q}
=
\left\{
(e,u)\in \mathbb G_m^2
\ \middle|\
c(ex)=e^n c(x),\ q(ex,uy)=u^d q(x,y)
\right\}.
\]

We also use the description of $\Aut(A_{c,q})$ obtained in \cite[Theorem 7]{ABEG}. In reduced form, every automorphism is described by a unipotent part and a diagonal part encoded by an element of $G_{c,q}$. In
particular, the unipotent subgroup is
\[
U_{\K}(A_{c,q})
=
\left\{
\exp(\ell(\bar x)\xi_{c,q})
\ \middle|\
\ell(x)\in\K[x]
\right\}.
\]
For such an automorphism, one has
\[
\exp(\ell(\bar x)\xi_{c,q})(\bar x)=\bar x
\quad \text{and} \quad 
\exp(\ell(\bar x)\xi_{c,q})(\bar y)
=
\bar y+c(\bar x)\ell(\bar x).
\]

The structural consequence of \cite[Theorem 7]{ABEG} is that there is a natural map
\[
\Aut(A_{c,q})\longrightarrow G_{c,q},
\]
whose kernel is $U_{\K}(A_{c,q})$. Equivalently, every automorphism can be
written as a product of a unipotent automorphism and an automorphism whose diagonal part is encoded by an element of $G_{c,q}$.

The following result from \cite[Theorem 8]{ABEG} provides structural motivation for the reduction to the diagonal group. It shows that, for non-locally nilpotent derivations, the isotropy group is controlled by the diagonal group $G_{c,q}$, up to a possible unipotent part.

\begin{theorem}[{\cite[Theorem 8]{ABEG}}]\label{thm:non-lnd-isotropy}
Let $D\in \Der(A_{c,q})$ be a non-locally nilpotent derivation. Then $\Aut_D(A_{c,q})$ is a linear algebraic group of dimension at most $3$. Moreover, it falls
into one of the following two cases:
\begin{enumerate}
    \item $\Aut_D(A_{c,q})$ is isomorphic to a closed subgroup of
    $G_{c,q}$;
    \item $\Aut_D(A_{c,q})$ is a semidirect product of a subgroup
    isomorphic to $\mathbb G_a$ and a closed subgroup of $G_{c,q}$.
\end{enumerate}
\end{theorem}

Let $A$ be a $\K$-algebra and let $D\in\Der(A)$. Recall that $D$ is called \textit{simple} if the only ideals $I\subseteq A$ such that $D(I)\subseteq I$ are $I=0$ and $I=A$. We denote by
\[
\Aut_D(A)
=
\{\sigma\in\Aut(A)\mid \sigma D=D\sigma\}
\]
the isotropy group of $D$.

We record the fixed-point argument that will be used throughout the paper. The key point is that, if $\sigma$ commutes with $D$, then the ideal generated by the elements $\sigma(a)-a$ is $D$-stable. For a simple derivation, this forces any element of the isotropy with a fixed point to be
the identity.

\begin{lemma}
Let $D\in\Der(A)$, and let $\sigma\in\Aut_D(A)$. Denote
\[
I_\sigma
:=
(\sigma(a)-a\mid a\in A).
\]

Then, $I_\sigma$ is a $D$-stable ideal of $A$.
\end{lemma}

\begin{proof}
For every $a\in A$, we have $D(\sigma(a)-a)=D\sigma(a)-D(a)$. Since $\sigma D=D\sigma$, $D(\sigma(a)-a)
=\sigma D(a)-D(a)$. Thus, $D(\sigma(a)-a)\in I_\sigma$ and
$D$ sends the generators of $I_\sigma$ into $I_\sigma$. By the Leibniz rule, for every $b\in A$,
\[
D\bigl(b(\sigma(a)-a)\bigr)
=
D(b)(\sigma(a)-a)+bD(\sigma(a)-a)\in I_\sigma.
\]

Therefore, by linearity, we obtain $D(I_\sigma)\subseteq I_\sigma$.

\end{proof}

\begin{corollary}
Let $D \in \Der(A)$ be simple. If $\sigma\in\Aut_D(A)$, then either $\sigma=\id$ or $I_\sigma=A$.
\end{corollary}

\begin{proof}
This follows immediately from the previous lemma and from the simplicity of $D$.
\end{proof}

The geometric meaning is the following: if $D$ is simple and
$\sigma\neq \id$ belongs to $\Aut_D(A)$, then $\sigma$ fixes no maximal ideal of $A$. Indeed, if $\sigma$ fixed a maximal ideal $\mathfrak m$, then
\[
\sigma(a)-a\in \mathfrak m
\]
for every $a\in A$, and hence $I_\sigma\subseteq \mathfrak m$. Thus, $I_\sigma\neq A$, forcing $I_\sigma=0$, and therefore, $\sigma=\id$.

The preceding discussion gives the following fixed-point criterion:

\begin{proposition}\label{fixpoint} Let $D\in\Der(A)$ be simple and let $\sigma\in\Aut_D(A)$. If $\sigma$ fixes a maximal ideal of $A$, then $\sigma=\id$.
\end{proposition}

We record that simple derivations are not locally nilpotent. This places them in the range of Theorem \ref{thm:non-lnd-isotropy}.

\begin{lemma}
Let $A=A_{c,q}$. If $D\in\Der(A)$ is simple, then $D$ is not locally
nilpotent.
\end{lemma}

\begin{proof}
Suppose, by contradiction, that $D$ is locally nilpotent. Since $\operatorname{ML}(A)=\K[\bar x]$, we have $D(\bar x)=0$. Let $\lambda\in\K$ such that $c(\lambda)\neq 0$. Then, the ideal $I=(\bar x-\lambda)$ is $D$-stable, because
$D(\bar x-\lambda)=0$. Indeed, for every $a\in A$,
\[
D\bigl(a(\bar x-\lambda)\bigr)
=
D(a)(\bar x-\lambda)+aD(\bar x-\lambda)
=
D(a)(\bar x-\lambda)\in I.
\]

Therefore, $I$ is a nonzero proper $D$-stable ideal of $A$, contradicting the simplicity of $D$, which proves the claim.

\end{proof}

\section{Diagonal reduction and fixed-point criteria}\label{3}

Throughout this section, we assume that $A_{c,q}$ is in reduced form. Furthermore, by the structure theorem for the automorphism group \cite[Lemma 6 and Theorem 7]{ABEG}, we have the split exact sequence
\[
1
\longrightarrow
U_{\K}(A_{c,q})
\longrightarrow
\Aut(A_{c,q})
\overset{\psi}{\longrightarrow}
G_{c,q}
\longrightarrow
1,
\]
where $\ker\psi=U_{\K}(A_{c,q})$.

\begin{proposition}\label{thm:unipotent-part-trivial}
Let $D\in\Der(A_{c,q})$ be a simple derivation. Then
\[
\Aut_D(A_{c,q})\cap U_{\K}(A_{c,q})
=
\{\id\}.
\]
\end{proposition}

\begin{proof}
Suppose that there exists a nontrivial element $\sigma\in \Aut_D(A_{c,q})\cap U_{\K}(A_{c,q})$. Then, we write $\sigma=\exp(\eta)$, with $\eta=\ell(\bar x)\xi_{c,q}$ for some nonzero $\ell(x)\in\K[x]$. Since $\sigma$ commutes with $D$, \cite[Lemma 11]{ABEG} implies that there exist $g(x),a(x),b(x)\in\K[x]$ such that
\[
D(\bar x)=g(\bar x) \quad \text{and} \quad D(\bar y)=a(\bar x)\bar y+b(\bar x).
\]

We use again that $\sigma D=D\sigma$ and $\sigma(\bar y)=\bar y+\ell(\bar x)c(\bar x)$, we have
\[
\sigma D(\bar y)=D\sigma(\bar y).
\]

The left-hand side is
\[
\sigma\bigl(a(\bar x)\bar y+b(\bar x)\bigr)
=
a(\bar x)\bar y+b(\bar x)+a(\bar x)\ell(\bar x)c(\bar x),
\]
and the right-hand side is
\[
D\bigl(\bar y+\ell(\bar x)c(\bar x)\bigr)
=
a(\bar x)\bar y+b(\bar x)+g(\bar x)(\ell c)'(\bar x).
\]

Therefore, $g(x)(\ell c)'(x)=a(x)\ell(x)c(x)$. We denote $F(x)=\ell(x)c(x)$. Since $F$ is nonconstant, $F$ has a root $\lambda\in\K$. Then, $g(\lambda)=0$. Thus,
\[
D(\bar x-\lambda)=g(\bar x)\in(\bar x-\lambda).
\]

Consequently, the nonzero proper ideal $I=(\bar x-\lambda)$ is $D$-stable, contradicting the simplicity of $D$.

\end{proof}

We use the canonical homomorphism $\psi:\Aut(A_{c,q})\longrightarrow G_{c,q}$ to obtain the following immediate consequence.

\begin{corollary}\label{cor:isotropy-injects-Gcq}
Let $D\in\Der(A_{c,q})$ be a simple derivation. Then the restriction
\[
\psi|_{\Aut_D(A_{c,q})}:
\Aut_D(A_{c,q})\longrightarrow G_{c,q}
\]
is injective. In particular, $\Aut_D(A_{c,q})$ is isomorphic to a subgroup of $G_{c,q}$.
\end{corollary}

\begin{proof}
Since $\ker\psi=U_{\K}(A_{c,q})$, we obtain
\[
\ker\bigl(\psi|_{\Aut_D(A_{c,q})}\bigr)
=
\Aut_D(A_{c,q})\cap U_{\K}(A_{c,q}).
\]

By Proposition \ref{thm:unipotent-part-trivial},
\[
\Aut_D(A_{c,q})\cap U_{\K}(A_{c,q})
=
\{\id\}.
\]

Therefore, $\ker\bigl(\psi|_{\Aut_D(A_{c,q})}\bigr)=\{\id\}$, and the restriction is injective.
\end{proof}

\begin{corollary}\label{cor:Gcq-trivial-implies-isotropy-trivial}
Let $D\in\Der(A_{c,q})$ be simple. If $G_{c,q}=\{1\}$, then
\[
\Aut_D(A_{c,q})=\{\id\}.
\]
\end{corollary}

\begin{proof}
This follows immediately from Corollary \ref{cor:isotropy-injects-Gcq}.
\end{proof}

The previous results show that the only possible nontrivial elements of $\Aut_D(A_{c,q})$ must project to nontrivial elements of $G_{c,q}$ under $\psi$. The following criterion explains how to rule out such elements using
fixed points.

\begin{proposition}\label{prop:fixed-point-criterion-via-psi}
Let $D\in\Der(A_{c,q})$ be a simple derivation. Suppose that every automorphism $\sigma\in\Aut(A_{c,q})$ such that $\psi(\sigma)\neq(1,1)$ fixes a maximal ideal of $A_{c,q}$. Then, 
\[\Aut_D(A_{c,q})=\{\id\}.\]
\end{proposition}

\begin{proof}
Let $\sigma\in\Aut_D(A_{c,q})$. If $\psi(\sigma)=(1,1)$, then
$\sigma\in\ker\psi=U_{\K}(A_{c,q})$. Thus, $\sigma\in\Aut_D(A_{c,q})\cap U_{\K}(A_{c,q})$. By Proposition \ref{thm:unipotent-part-trivial}, we obtain $\sigma=\id$.

Suppose that $\psi(\sigma)\neq(1,1)$. By hypothesis, $\sigma$ fixes a maximal ideal of $A_{c,q}$. Since $\sigma\in\Aut_D(A_{c,q})$ and $D$ is simple, the Proposition \ref{fixpoint} implies that $\sigma=\id$. Therefore, 
we obtain the desired result.
\end{proof}

\section{The generic case}\label{4}

In this section, we obtain a criterion for the triviality of the diagonal group $G_{c,q}$ in terms of the monomials appearing in $c$ and $q$. Recall that, throughout the paper, $A_{c,q}$ is assumed to be in reduced form. In particular, $\deg_x(q)<\deg(c)$ and also we write
\[
n=\deg(c),
\quad
d=\deg_y(q).
\]

Let us also recall what a \textit{general} reduced pair  means. Fix integers $n\geq2$ and $d\geq2$. Since $c$ is centered, we may write
\[
c(x)=\sum_{r=0}^{n}c_rx^r,
\quad
c_n\neq0,
\quad
c_{n-1}=0.
\]

Since $q$ is quasi-monic of degree $d$ with respect to $y$, centered as a polynomial in $\K[x,y]$, and satisfies $\deg_x(q)<n$, we may write
\[
q(x,y)
=
\lambda y^d
+
\sum_{\substack{0\leq i<n\\0\leq j\leq d-2}}
q_{ij}x^iy^j,
\quad
\lambda\in\K^\ast.
\]

Thus, for fixed $n$ and $d$, the pairs form a finite-dimensional parameter space with coordinates given by the coefficients of $c$ and $q$, subject to
$c_n\neq0$ and $\lambda\neq0$. When we say that a property holds for a \textit{general} reduced pair $(c,q)$, we
mean that it holds on a nonempty Zariski open subset of this parameter space.

In reduced form, the diagonal group is
\[
G_{c,q}
=
\left\{
(e,u)\in \mathbb G_m^2
\ \middle|\
c(ex)=e^n c(x),\ q(ex,uy)=u^d q(x,y)
\right\}.
\]

With the notation above, let $\Lambda(c,q)$ denote the subgroup of $\mathbb Z^2$ generated by the vectors $(n-r,0)$, for all $r<n$ such that $c_r\neq0$, together with the vectors $(i,j-d)$, for all pairs $(i,j)$ such that $q_{ij}\neq0$.

\begin{proposition}\label{prop:lattice-criterion-Gcq}
If $\Lambda(c,q)=\mathbb Z^2$, then $G_{c,q}=\{1\}$.
\end{proposition}

\begin{proof}
Let $(e,u)\in G_{c,q}$. Then,
\[
c(ex)=e^n c(x)
\quad \text{and} \quad
q(ex,uy)=u^d q(x,y).
\]

From the first identity, for every monomial $x^r$ appearing in $c(x)$ with $r<n$, we obtain $e^r=e^n$, and then $e^{n-r}=1$. From the second identity, for every monomial $x^iy^j$ appearing in $q(x,y)$, we obtain $e^iu^j=u^d$, and then $e^iu^{j-d}=1$. 

Thus, for every generator $(a,b)$ of $\Lambda(c,q)$, we have $e^a u^b=1$. Since $\Lambda(c,q)=\mathbb Z^2$, there exist generators $(a_1,b_1),\dots,(a_s,b_s)$
of $\Lambda(c,q)$ and integers $m_1,\dots,m_s\in\mathbb Z$
such that
\[
(1,0)=\sum_{k=1}^s m_k(a_k,b_k).
\]

Therefore,
\[
e
=
e^1u^0
=
e^{\sum_{k=1}^s m_ka_k}
u^{\sum_{k=1}^s m_kb_k}
=
\prod_{k=1}^s
\left(e^{a_k}u^{b_k}\right)^{m_k}
=
1.
\]

Similarly, there exist integers $n_1,\dots,n_s\in\mathbb Z$
such that $(0,1)=\sum_{k=1}^s n_k(a_k,b_k)$.
Thus,
\[
u
=
e^0u^1
=
e^{\sum_{k=1}^s n_ka_k}
u^{\sum_{k=1}^s n_kb_k}
=
\prod_{k=1}^s
\left(e^{a_k}u^{b_k}\right)^{n_k}
=
1.
\]

Therefore, $(e,u)=(1,1)$ and then, $G_{c,q}=\{1\}$.
\end{proof}

\begin{corollary}\label{cor:lattice-trivial-isotropy}
If $\Lambda(c,q)=\mathbb Z^2$, then every simple derivation $D\in\Der_{\K}(A_{c,q})$ satisfies
\[
\Aut_D(A_{c,q})=\{\id\}.
\]
\end{corollary}

\begin{proof}
By Proposition \ref{prop:lattice-criterion-Gcq}, we have $G_{c,q}=\{1\}$. Therefore, the result follows from
Corollary \ref{cor:Gcq-trivial-implies-isotropy-trivial}.
\end{proof}

Assume that $d\geq3$. We denote by $\Omega_{n,d}$ the subset of the parameter space of reduced pairs of fixed degrees $(n,d)$ given by
\[
\Omega_{n,d}
=
\left\{
(c,q)
\ \middle|\
q_{00}q_{01}q_{10}\neq0
\right\}.
\]

We use the lattice criterion to exhibit an explicit nonempty Zariski open subset in degree $d\geq3$:

\begin{proposition}\label{prop:generic-full-lattice}
The subset $\Omega_{n,d}$ is a nonempty Zariski open subset. Moreover, for every $(c,q)\in\Omega_{n,d}$, we have
\[
\Lambda(c,q)=\mathbb Z^2.
\]
\end{proposition}

\begin{proof}
The condition $q_{00}q_{01}q_{10}\neq0$ defines a nonempty Zariski open subset of the parameter space. Let $(c,q)\in\Omega_{n,d}$. Then, the monomials $1, y$ and $x$
appear in $q(x,y)$ with nonzero coefficients.

The monomial $1=x^0y^0$ contributes with the vector $(0,-d)$,
and the monomial $y=x^0y^1$ contributes with the vector $(0,1-d)$. Since $\gcd(d,d-1)=1$, these two vectors generate
$\{0\}\times\mathbb Z$. The monomial $x=x^1y^0$ contributes with the vector $(1,-d)$. Together with $\{0\}\times\mathbb Z$, this vector generates the whole lattice $\mathbb Z^2$. Thus, $\Lambda(c,q)=\mathbb Z^2$. By Proposition \ref{prop:lattice-criterion-Gcq}, we conclude that $G_{c,q}=\{1\}$.
\end{proof}

\begin{corollary}\label{cor:generic-trivial-isotropy}
Let $d=\deg_y(q)\geq3$. For every reduced pair
$(c,q)\in\Omega_{n,d}$, every simple derivation $
D\in\Der_{\K}(A_{c,q})$ satisfies
\[
\Aut_D(A_{c,q})=\{\id\}.
\]
\end{corollary}

\begin{proof}
By Proposition \ref{prop:generic-full-lattice}, for every
$(c,q)\in\Omega_{n,d}$, we have $\Lambda(c,q)=\mathbb Z^2$. Therefore, the result follows from
Corollary \ref{cor:lattice-trivial-isotropy}.
\end{proof}

\begin{remark}
The condition $q_{00}q_{01}q_{10}\neq0$ is only a convenient sufficient condition ensuring $\Lambda(c,q)=\mathbb Z^2$. There may be reduced pairs outside $\Omega_{n,d}$ for which $\Lambda(c,q)=\mathbb Z^2$ still holds. For all such pairs, Corollary \ref{cor:lattice-trivial-isotropy}
also applies.
\end{remark}

\begin{remark}
The assumption $d\geq3$ is necessary for this particular argument. If $d=2$ and $q$ is centered as a polynomial in $\K[x,y]$, then
\[
q(x,y)=\lambda y^2+p(x),
\quad
\lambda\in\K^\ast.
\]

Therefore, $q(x,-y)=q(x,y)$ and, consequently, $(1,-1)\in G_{c,q}$.
\end{remark}

\section{The quadratic case}\label{5}

In this section, we now discuss the exceptional case $d=\deg_y(q)=2$, where the centered form of $q$ always gives the nontrivial element $(1,-1)\in G_{c,q}$. Since $q$ is centered as a polynomial in $\K[x,y]$, we write $q(x,y)=\lambda y^2+p(x)$, $\lambda\in\K^\ast$, $p(x)=p_0+p_1x+\cdots+p_{n-1}x^{n-1}\in\K[x]$, with $\deg(p)<n=\deg(c)$. 

Thus,
\[
A_{c,q}
=
\K[x,y,z]/(c(x)z-\lambda y^2-p(x)).
\]

Recall that,
\[
G_{c,q}
=
\left\{
(e,u)\in\mathbb G_m^2
\ \middle|\
c(ex)=e^n c(x),\ q(ex,uy)=u^2q(x,y)
\right\}.
\]

Since $q(ex,uy)=\lambda u^2y^2+p(ex)$ and $u^2q(x,y)=\lambda u^2y^2+u^2p(x)$, the second defining equation becomes $p(ex)=u^2p(x)$. Therefore,
\[
G_{c,q}
=
\left\{
(e,u)\in\mathbb G_m^2
\ \middle|\
c(ex)=e^n c(x),\ p(ex)=u^2p(x)
\right\}.
\]

\begin{proposition}\label{prop:quadratic-generic-Gcq}
Let $d=\deg_y(q)=2$. Suppose that $p_0p_1\neq0$. Then,
\[
G_{c,q}=\{(1,1),(1,-1)\}\cong \mathbb Z/2\mathbb Z.
\]
\end{proposition}

\begin{proof}
Let $(e,u)\in G_{c,q}$. Then, $p(ex)=u^2p(x)$. Comparing the constant term and using $p_0\neq0$, we obtain $p_0=u^2p_0$. Thus, $u^2=1$.

Comparing the coefficient of $x$, and using $p_1\neq0$, we obtain $ep_1=u^2p_1$. Since $u^2=1$, it follows that
$e=1$. Thus, every element of $G_{c,q}$ is of the form $(1,u)$ with $u^2=1$. Therefore,
\[
G_{c,q}=\{(1,1),(1,-1)\}.
\]
\end{proof}

In the generic quadratic case, the only nontrivial element of $G_{c,q}$ is $(1,-1)$. An automorphism of $A_{c,q}$ with image $(1,-1)$ under $\psi$ may have a
unipotent component. The next lemma shows that this does not affect the fixed-point argument.

\begin{lemma}\label{lem:quadratic-lifts-fixed-point}
Let $d=\deg_y(q)=2$ and $G_{c,q}=\{(1,1),(1,-1)\}$. Let
\[
\psi:\Aut(A_{c,q})\longrightarrow G_{c,q}
\]
be the canonical homomorphism. If $\sigma\in\Aut(A_{c,q})$ satisfies $\psi(\sigma)=(1,-1)$, then $\sigma$ fixes a maximal ideal of $A_{c,q}$.
\end{lemma}

\begin{proof}
By the structure theorem for $\Aut(A_{c,q})$  \cite[Lemma 6 and Theorem 7]{ABEG}, every automorphism
$\sigma\in\Aut(A_{c,q})$ can be written in the form
\[
\sigma(\bar x)=e\bar x
\quad \text{and} \quad
\sigma(\bar y)=u\bar y+c(\bar x)\ell(\bar x),
\]
for some $(e,u)\in G_{c,q}$ and some $\ell(x)\in\K[x]$. Moreover, $\psi(\sigma)=(e,u)$. Since, in our case,
$\psi(\sigma)=(1,-1)$, we obtain $e=1$
and $u=-1$. Therefore,
\[
\sigma(\bar x)=\bar x
\quad \text{and} \quad
\sigma(\bar y)=-\bar y+c(\bar x)\ell(\bar x).
\]

Let $\alpha\in\K$ such that $c(\alpha)\neq0$. We denote
\[
\mu=\frac{c(\alpha)\ell(\alpha)}{2}
\quad \text{and} \quad
\nu=\frac{\lambda\mu^2+p(\alpha)}{c(\alpha)}.
\]

Then, $c(\alpha)\nu-\lambda\mu^2-p(\alpha)=0$. Thus, $\mathfrak m=
(\bar x-\alpha,\bar y-\mu,\bar z-\nu)$ is a maximal ideal of $A_{c,q}$. By the choice of $\mu$ and $\nu$, it is straightforward to check that $\sigma(\mathfrak m)\subseteq\mathfrak m$. Therefore, since $\sigma$ is an automorphism and $\mathfrak m$ is maximal, $\sigma(\mathfrak m)=\mathfrak m$.

\end{proof}

\begin{theorem}\label{thm:quadratic-generic-trivial-isotropy}
Let $d=\deg_y(q)=2$ and, with the notation above, suppose that $p_0p_1\neq0$. Then, every simple derivation $D\in\Der_{\K}(A_{c,q})$ satisfies
\[
\Aut_D(A_{c,q})=\{\id\}.
\]
\end{theorem}

\begin{proof}
Let $D$ be a simple derivation of $A_{c,q}$, and let $\sigma\in\Aut_D(A_{c,q})$. By Proposition \ref{prop:quadratic-generic-Gcq}, we have $G_{c,q}=\{(1,1),(1,-1)\}$. Consider the canonical homomorphism
\[
\psi:\Aut(A_{c,q})\longrightarrow G_{c,q}.
\]

If $\psi(\sigma)=(1,1)$, then $\sigma\in\ker\psi=U_{\K}(A_{c,q})$. Then, $\sigma\in\Aut_D(A_{c,q})\cap U_{\K}(A_{c,q})$. By the triviality of the unipotent part in the isotropy group of a simple derivation, Proposition \ref{thm:unipotent-part-trivial}, we obtain $\sigma=\id$. Suppose that $\psi(\sigma)=(1,-1)$. By Lemma \ref{lem:quadratic-lifts-fixed-point}, the automorphism $\sigma$ fixes a maximal ideal of $A_{c,q}$. Since $D$ is simple and $\sigma\in\Aut_D(A_{c,q})$,
the fixed-point criterion, Proposition \ref{fixpoint}, implies that $\sigma=\id$. Therefore, $\Aut_D(A_{c,q})=\{\id\}$.
\end{proof}

\begin{corollary}\label{cor:generic-all-degrees}
For each fixed pair of degrees $n=\deg(c)\geq2$ and $d=\deg_y(q)\geq2$, there exists a nonempty Zariski open subset of the parameter space of reduced pairs $(c,q)$ such that, for every pair in this open subset and every simple derivation $D\in\Der_{\K}(A_{c,q})$, one has
\[
\Aut_D(A_{c,q})=\{\id\}.
\]
\end{corollary}

\begin{proof}
If $d\geq3$, the result follows from Corollary \ref{cor:generic-trivial-isotropy}. If
$d=2$, then we write $q(x,y)=\lambda y^2+p(x)$. In this case, we consider the nonempty Zariski open subset defined by
$p_0p_1\neq0$. For pairs in this open subset, the result follows from Theorem \ref{thm:quadratic-generic-trivial-isotropy}.
\end{proof}

\begin{remark}
The quadratic case shows that the triviality of $G_{c,q}$ is not necessary for the triviality of the isotropy group of a simple derivation. Indeed, on the open subset defined by
$p_0p_1\neq0$, we have
\[
G_{c,q}=\{(1,1),(1,-1)\}\cong\mathbb Z/2\mathbb Z.
\]
However, every simple derivation $D\in\Der_{\K}(A_{c,q})$
satisfies $\Aut_D(A_{c,q})=\{\id\}$. The reason is that every automorphism whose image under
$\psi:\Aut(A_{c,q})\longrightarrow G_{c,q}$ is the nontrivial element $(1,-1)$ fixes a maximal ideal of $A_{c,q}$. By the fixed-point criterion, such an automorphism cannot belong to $\Aut_D(A_{c,q})$, when $D$ is simple.

Thus, any counterexample to the statement $\Aut_D(A_{c,q})=\{\id\}$ must come from a special reduced pair for which some nontrivial element of $G_{c,q}$ admits a lift acting without fixed maximal ideals.
\end{remark}

The next example shows that this special phenomenon actually occurs.

\begin{example}\label{ex:special-fixed-point-free-diagonal}
Consider the Danielewski-type algebra
\[
\mathbb A
=
\K[x,y,z]/(x^2z-y^2-1).
\]
Thus, $c(x)=x^2$ and $q(x,y)=y^2+p(x)$ with $p(x)=1$. In particular, this pair does not belong to the open subset defined by $p_0p_1\neq0$. The diagonal automorphism corresponding to $(-1,-1)\in G_{c,q}$ is given by
\[
\sigma(\bar x)=-\bar x,
\quad
\sigma(\bar y)=-\bar y,
\quad
\sigma(\bar z)=\bar z.
\]

Furthermore, this automorphism has no fixed maximal ideal. Indeed, a fixed point would have to satisfy that $\bar x=0$ and $\bar y=0$. Substituting these values into the defining equation we obtain $-1=0$, which is impossible.

Therefore, this special pair lies outside the generic quadratic open subset
treated above, and the fixed-point criterion does not exclude the diagonal
automorphism $\sigma$. In the next section, we show that this phenomenon actually produces a simple derivation with nontrivial isotropy.
\end{example}

\section{A counterexample} \label{6}

Consider
\[
\mathbb A=\K[x,y,z]/(x^2z-y^2-1)
\]
and the involution $\sigma(\bar x)=-\bar x$, $\sigma(\bar y)=-\bar y$, and $\sigma(\bar z)=\bar z$. Thus, $\sigma$ has no fixed point on $\Spec(\mathbb A)$. 

We denote by
\[
\mathbb A^{\langle\sigma\rangle}
=
\{a\in \mathbb A\mid \sigma(a)=a\}
\]
the invariant ring of the involution $\sigma$. Notice that $\mathbb A^{\langle\sigma\rangle}$ is generated by
\[
u=\bar x^2,
\quad
v=\bar x\bar y,
\quad
w=\bar z.
\]

Moreover, $v^2=\bar x^2\bar y^2=u(uw-1)=u^2w-u$. Therefore,
\[
\mathbb A^{\langle\sigma\rangle}
\cong
\mathbb B
=
\K[u,v,w]/(u^2w-v^2-u).
\]

Using these notations, we obtain the following lemma.

\begin{lemma}\label{lem:explicit-lift-involution}
Let $E\in\Der_{\K}(\mathbb B)$. Then, there exists a unique derivation $D\in\Der_{\K}(\mathbb A)$ such that $D|_{\mathbb B}=E$. Moreover, this derivation commutes with $\sigma$ and also if $E$ is simple, then $D$ is simple.
\end{lemma}

\begin{proof}
Write
\[
E_u=E(u),
\quad
E_v=E(v),
\quad
E_w=E(w).
\]
Since $E$ preserves the relation $u^2w-v^2-u=0$,
we have
\[
(2uw-1)E_u+u^2E_w-2vE_v=0.
\]

We define
\[
D(\bar z)=E_w,
\quad
D(\bar x)=\frac{E_u}{2\bar x},
\quad
D(\bar y)=\frac{E_v-\bar yD(\bar x)}{\bar x}.
\]
These expressions belong to $\mathbb A$. Indeed, using the identity above and the equations $u=\bar x^2$, $v=\bar x\bar y$, and $w=\bar z$, one checks that $E_u\in \bar x\mathbb A$ and $E_v-\bar yD(\bar x)\in \bar x\mathbb A$. By construction,
\[
D(\bar x^2)=E_u,
\quad
D(\bar x\bar y)=E_v,
\quad
D(\bar z)=E_w.
\]

Then,
\[
D(u)=E(u),
\quad
D(v)=E(v),
\quad
D(w)=E(w),
\]
so $D|_{\mathbb B}=E$.

It is straightforward to check that $D$ is well defined on $\mathbb A$ and the uniqueness follows from the identities $u=\bar x^2$, $v=\bar x\bar y$, and $w=\bar z$. 

We show that $D$ commutes with $\sigma$. Since
\[
E_u,E_v,E_w\in\mathbb B=\mathbb A^{\langle\sigma\rangle},
\]
they are invariant under $\sigma$. From the expressions defining $D$, we obtain that $D(\bar x)$ and $D(\bar y)$ are anti-invariant, while $D(\bar z)$ is invariant. Thus,
\[
\sigma D(\bar x)=D\sigma(\bar x),
\quad
\sigma D(\bar y)=D\sigma(\bar y),
\quad
\sigma D(\bar z)=D\sigma(\bar z).
\]
Therefore, $\sigma D=D\sigma$.

Finally, assume that $E$ is simple. Let $0\neq I\subseteq \mathbb A$ be a $D$-stable ideal. Let $0\neq a\in I$. Since $\mathbb A$ is a
domain,
\[
N(a):=a\sigma(a)\neq0.
\]
Moreover, $N(a)\in \mathbb A^{\langle\sigma\rangle}=\mathbb B$ and, because $a\in I$, also $N(a)\in I$. Thus, $0\neq I\cap\mathbb B$. The ideal $J=I\cap\mathbb B$ is $E$-stable, since $D|_{\mathbb B}=E$. By the simplicity of $E$, we have $J=\mathbb B$. Therefore, $1\in J\subseteq I$,
and hence $I=\mathbb A$. Thus, $D$ is simple.
\end{proof}

\begin{remark}

Using the notations above, suppose that $ \mathbb B$ admits a simple derivation $E\in\Der_{\K}(\mathbb B)$. By Lemma \ref{lem:explicit-lift-involution}, $E$ lifts to a simple derivation $D\in\Der_{\K}(\mathbb A)$ such that $\sigma D=D\sigma$. Therefore, $\sigma\in\Aut_D(\mathbb A)$. Since
$\sigma\neq\id$, we obtain
\[
\Aut_D(\mathbb A)\neq\{\id\}.
\]

Thus, if the quotient algebra
\[
\mathbb B=\K[u,v,w]/(u^2w-v^2-u)
\]
admits a simple derivation, then the analogue of the Mendes-Pan theorem fails for the Danielewski-type algebra
\[
\mathbb A=\K[x,y,z]/(x^2z-y^2-1).
\]
    
\end{remark}

Thus the search for a counterexample is reduced, in this case, to the following concrete problem:

\begin{question}
Does the algebra $\mathbb B$ admit a simple derivation?
\end{question}

We prove that the answer is affirmative by constructing an explicit simple derivation of $\mathbb B$.

\begin{proposition}\label{B-admits-simple-derivation}
The derivation $\mathcal{D}\in\Der_{\K}(\mathbb B)$ defined by
\[
\mathcal{D} (w)=1,
\]
\[
\mathcal{D}(u)=v+u^2w-u-u^2,
\]
\[
\mathcal{D}(v)=uvw-uv+uw-\frac{v+1}{2}
\]
is simple.
\end{proposition}

\begin{proof}
It is straightforward to check that $\mathcal{D}$ is well defined on $\mathbb B$. Indeed, 
\[
\mathcal{D}(u^2w-v^2-u)=(2uw-2u-1)(u^2w-v^2-u).
\]

We localize at $u$. In $\mathbb B_u$, denote $t:=\frac vu$. Using the relation $u^2w=v^2+u$, we obtain
\[
\frac1u=w-t^2.
\]

Thus, $\mathbb B_u\cong \K[w,t,(w-t^2)^{-1}]$. Under this identification, $\mathcal{D}$ is the localization of the derivation
\[
\delta
=
\partial_w+\frac{w+t-t^2}{2}\partial_t
\]
of $\K[w,t]$. Indeed,
\[
\mathcal{D}(w)=1
\]
and
\[
\mathcal{D}(t)
=
\mathcal{D}\left(\frac vu\right)
=
\frac{u\mathcal{D}(v)-v\mathcal{D}(u)}{u^2}=
\frac{w+t-t^2}{2}.
\]

We claim that $\delta$ is simple. Denote
\[
X:=w+\frac14,
\quad
T:=t-\frac12.
\]

Then, $\delta(X)=1$ and $\delta(T)=\frac{X-T^2}{2}$. Thus,
\[
\delta
=
\partial_X+\frac{X-T^2}{2}\partial_T.
\]

Since $\K$ is algebraically closed and has characteristic zero, we can choose $a\in\K^\ast$ such that $4a^3=-1$. Denote,
\[
b:=2a^2,
\quad
x:=aX,
\quad
y:=bT.
\]

Then, $\frac1a\delta(x)=1$ and $\frac1a\delta(y)=y^2+x$. Therefore,
\[
\frac1a\delta
=
\partial_x+(y^2+x)\partial_y.
\]

Notice that the derivation $\partial_x+(y^2+x)\partial_y$
is Nowicki's simple derivation; see \cite{Nowicki2008}. Therefore, $\delta$ is simple also. It follows that the localization
\[
\K[w,t,(w-t^2)^{-1}]
\]
is also $\delta$-simple. 

Let $0\neq I\subseteq \mathbb B$ be an $\D$-stable ideal. Since $\mathbb B$ is a domain, the localization $I_u\subseteq \mathbb B_u$ is a nonzero $\D$-stable ideal. Since $\mathbb B_u$ is $\D$-simple, we obtain $I_u= \mathbb B_u$.
Therefore, there exists $m\geq0$ such that $u^m\in I$: choose $m$ minimal with this property.

We prove that $m=0$. Suppose, by contradiction, that $m>0$. Since $u^m\in I$ and $I$ is $\D$-stable, we have $\D(u^m)\in I$. However,
\[
\D(u^m)=m u^{m-1}v+m u^{m+1}w-m u^m-m u^{m+1}.
\]
The last three terms belong to $I$, it follows that $u^{m-1}v\in I$. If $m=1$, then $u\in I$ and $v\in I$.

Since $\D(v)=uvw-uv+uw-\frac{v+1}{2}$ and $\D(v)\in I$, we obtain $I= \mathbb B$: this contradicts the minimal choice of $m>0$. Therefore, $m\neq1$.

Assume that $m\geq2$. Since $u^{m-1}v\in I$
and $I$ is $\D$-stable, we have $\D(u^{m-1}v)\in I$. By the Leibniz rule,
\[
\D(u^{m-1}v)
=
(m-1)u^{m-2}\D(u)v
+
u^{m-1}\D(v).
\]

Using the defining relation $v^2=u^2w-u$, we obtain $u^{m-2}v^2=u^mw-u^{m-1}$ and then:
\[
u^{m-2}\D(u)v
=
u^mw-u^{m-1}
+
u^mwv
-
u^{m-1}v
-
u^mv.
\]
All terms in this expression belong to $I$, except possibly $
-u^{m-1}$. Indeed, $u^m\in I$ and $u^{m-1}v\in I$. Therefore,
$(m-1)u^{m-2}\D(u)v
+
(m-1)u^{m-1}
\in I$. On the other hand, we have
\[
u^{m-1}\D(v)
=
u^mvw-u^mv+u^mw-\frac12u^{m-1}v-\frac12u^{m-1}.
\]
Again, all terms belong to $I$, except possibly $-\frac12u^{m-1}$. Thus, $u^{m-1}\D(v)+\frac12u^{m-1}\in I$.

Combining the two previous observations and using $\D(u^{m-1}v)\in I$, we obtain
\[
-\left(m-1\right)u^{m-1}-\frac12u^{m-1}\in I.
\]

That is, $-\left(m-\frac12\right)u^{m-1}\in I$ and it follows that $u^{m-1}\in I$: this again contradicts the minimality of $m$. Finally, $1=u^0\in I$, so $I= \mathbb B$.
Thus, every nonzero $\D$-stable ideal of $\mathbb B$ is equal to $\mathbb B$, and then $\D$ is simple.
\end{proof}

Combining Lemma \ref{lem:explicit-lift-involution} with
Proposition \ref{B-admits-simple-derivation}, we obtain a simple derivation
$D\in\Der_{\K}(\mathbb A)$ whose restriction to $\mathbb B$ is
$\mathcal D$. Moreover, this derivation commutes with the involution
\[
\sigma(\bar x)=-\bar x,
\quad
\sigma(\bar y)=-\bar y,
\quad
\sigma(\bar z)=\bar z.
\]
Thus, $\sigma\in\Aut_D(\mathbb A)$. Since $\sigma\neq\id$, it follows that $\Aut_D(\mathbb A)\neq\{\id\}$. Therefore, the Mendes-Pan analogue does not hold for all Danielewski-type algebras.

\section{Foliations and birational symmetries}\label{7}

In this section we assume that $\K=\mathbb C$. Let $X=\Spec(A_{c,q})$ and let $D\in\Der_{\mathbb C}(A_{c,q})$
be a simple derivation. The derivation $D$ defines an algebraic vector field on the affine surface $X$. Since $D$ is simple, this vector field has no zero on $X$. Indeed, if $p\in X$ were a zero of $D$, then the corresponding maximal ideal $\mathfrak m_p\subseteq A_{c,q}$ would be $D$-stable, contradicting simplicity. Similarly, $D$ does not preserve any algebraic curve contained in $X$. Indeed, if $C\subseteq X$ were an invariant algebraic curve, then its defining ideal
$I(C)\subseteq A_{c,q}$ would be a nonzero proper $D$-stable ideal. This again contradicts simplicity. Thus, on the smooth locus  $X_{\mathrm{reg}} $, the derivation  $D $ defines a nonsingular algebraic foliation  $\mathcal F_D $ with no invariant algebraic curve contained in  $X $.

This is the natural analogue, in the Danielewski-type setting, of the foliations associated to simple derivations of $\mathbb C[x,y]$. In the affine plane case, such foliations were studied by Cousin, Mendes and Pan
\cite{CMP}. Their work considers the foliation on
$\mathbb C^2$ defined by a simple derivation and its extension to $\mathbb P^2=\mathbb C^2\cup L_\infty$. They prove, among other things, that if the restriction of a foliation of $\mathbb P^2$ to $\mathbb C^2$ has no invariant algebraic curve, then the group of birational symmetries of the foliation is finite. In particular, this applies to foliations associated to simple derivations of
$\mathbb C[x,y]$.

There is, however, an important difference between the affine plane and the Danielewski-type setting. In the affine plane case, one has the compactification $\mathbb C^2\subseteq \mathbb P^2$ with boundary the line at infinity
$L_\infty=\mathbb P^2\setminus\mathbb C^2$. For a Danielewski-type surface, the choice of
compactification and the geometry of the boundary divisor are more subtle. Thus, the birational geometry of the foliation may depend on the interaction between the compactified foliation and the boundary.

This suggests that the correct analogue of the affine-plane theory should take into account not only the foliation, but also a compactification pair
\[
(\overline X,\Delta),
\]
where $X\subseteq\overline X$ is a smooth projective compactification and $\Delta=\overline X\setminus X$
is the boundary divisor.

There are several symmetry groups naturally associated to $D$ and to $\mathcal F_D$. The first one is the isotropy group of the derivation: $\Aut_D(A_{c,q})$, the group studied in the previous sections. A second, less restrictive group is the analogue, in our affine surface setting, of the group $\operatorname{Pol}(\mathcal F_\partial)$ considered
by Cousin, Mendes and Pan \cite{CMP}:
\[
\Pol(\mathcal F_D)
=
\left\{
\sigma\in\Aut_{\mathbb C}(X)
\ \middle|\
\sigma D\sigma^{-1}=h_\sigma D
\text{ for some }h_\sigma\in\mathcal O(X)^\ast
\right\}.
\]

Clearly, $\Aut_D(A_{c,q})\subseteq
\Pol(\mathcal F_D)$. This inclusion may be strict; see \cite[Proposition 7.1]{CMP}, where foliations associated to simple derivations with nontrivial polynomial symmetries are constructed.

Let $j:X\hookrightarrow \overline X$ be a smooth projective compactification with boundary divisor $\Delta=\overline X\setminus X$. The vector field $D$ extends to a rational vector field on $\overline X$. After removing divisorial zeroes and poles, it defines a foliation $\overline{\mathcal F}_D$ on $\overline X$.

We define $\Bir(\overline{\mathcal F}_D)$ to be the group of birational self-maps of $\overline X$ preserving the
foliation $\overline{\mathcal F}_D$. Thus, there are natural inclusions
\[
\Aut_D(A_{c,q})
\subseteq
\Pol(\mathcal F_D)
\subseteq
\Bir(\overline{\mathcal F}_D).
\]

The results proved in this paper concern the smallest of these groups, $\Aut_D(A_{c,q})$. However, the foliation viewpoint suggests further questions. Even when
\[
\Aut_D(A_{c,q})=\{\id\},
\]
the larger groups $\Pol(\mathcal F_D)$ and $\Bir(\overline{\mathcal F}_D)$ may be nontrivial. 

This motivates the following questions:

\begin{question}
Let $D$ be a simple derivation of $A_{c,q}$. Is the group $\Pol(\mathcal F_D)$ finite?
\end{question}

Since the boundary divisor plays a more important role in the
Danielewski-type setting, one may also formulate a relative version:

\begin{question}
Is the group $\Bir(\overline X,\Delta,\overline{\mathcal F}_D)$ of birational maps preserving both the foliation and the boundary finite?
\end{question}

\begin{question}
Which Kodaira dimensions can occur for the compactified foliations $\overline{\mathcal F}_D$ arising from simple derivations of Danielewski-type algebras?
\end{question}

\bigskip
\bigskip
\noindent
\textsc{Rene Baltazar}\\
Universidade Federal do Rio Grande - FURG\\
Santo Antônio da Patrulha/RS, Brasil\\
\texttt{renebaltazar.furg@gmail.com}


\begin{thebibliography}{99}

\bibitem{ABEG}
A. Ahouita, R. Baltazar, M. El Kahoui and S. Gaifullin,
\emph{The isotropy group of a derivation on a Danielewski-type algebra},
preprint, arXiv:2510.07059.

\bibitem{Baltazar2016}
R. Baltazar,
\emph{On simple Shamsuddin derivations in two variables},
Anais da Academia Brasileira de Ci\^encias \textbf{88} (2016), 2031--2038.

\bibitem{PanBalta}
R. Baltazar and I. Pan,
\emph{On the automorphism group of a polynomial differential ring in two variables},
Journal of Algebra \textbf{576} (2021), 197--227.

\bibitem{BianchiVeloso}
A. C. Bianchi and M. O. Veloso,
\emph{Locally nilpotent derivations and automorphism groups of certain Danielewski surfaces},
Journal of Algebra \textbf{469} (2017), 96--108.

\bibitem{CMP}
G. Cousin, L. G. Mendes and I. Pan,
\emph{Birational geometry of foliations associated to simple derivations},
Bull. Soc. Math. France \textbf{147} (2019), no. 4, 607--638.

\bibitem{Hart1975}
R. Hart,
\emph{Derivations on regular local rings of finitely generated type},
Journal of the London Mathematical Society \textbf{10} (1975), no. 3,
292--294.

\bibitem{MP}
L. G. Mendes and I. Pan,
\emph{On plane polynomial automorphisms commuting with simple derivations},
J. Pure Appl. Algebra \textbf{221} (2017), no. 4, 875--882.

\bibitem{Nowicki2008}
A. Nowicki,
\emph{An example of a simple derivation in two variables},
Colloquium Mathematicum \textbf{113} (2008), no. 1, 25--31.

\bibitem{Seidenberg1967}
A. Seidenberg,
\emph{Differential ideals in rings of finitely generated type},
American Journal of Mathematics \textbf{89} (1967), no. 1, 22--42.

\bibitem{Shamsuddin1977}
A. Shamsuddin,
\emph{Automorphisms and skew polynomial rings},
Ph.D. thesis, University of Leeds, 1977.

\bibitem{Shamsuddin1982}
A. Shamsuddin,
\emph{Rings with automorphisms leaving no nontrivial proper ideals invariant},
Canadian Mathematical Bulletin \textbf{25} (1982), no. 4, 478--486.

\end{thebibliography}
\end{document}